\documentclass[12pt,reqno]{amsart}
\usepackage{amssymb,amsmath,amsthm,enumerate,verbatim,bbm}

\usepackage{mathrsfs}
\usepackage[a4paper]{geometry}

\DeclareMathOperator{\sign}{sign}

\DeclareMathOperator{\Ker}{Ker}
\DeclareMathOperator{\Tr}{Tr}

\DeclareMathOperator{\supp}{supp}

\DeclareMathOperator{\meas}{meas}

\renewcommand{\Re}{\operatorname{Re}}

\newcommand{\abs}[1]{\lvert#1\rvert}
\newcommand{\Abs}[1]{\left\lvert#1\right\rvert}
\newcommand{\norm}[1]{\lVert#1\rVert}

\newcommand{\Norm}[1]{\left\lVert#1\right\rVert}
\newcommand{\jap}[1]{\langle#1\rangle}

\newcommand{\bbR}{{\mathbb R}}
\newcommand{\bbC}{{\mathbb C}}

\newcommand{\bbN}{{\mathbb N}}

\newcommand{\wh}{\widehat}

\newcommand{\calK}{{\mathcal K}}

\newcommand{\bS}{\mathbf{S}}

\numberwithin{equation}{section}

\theoremstyle{plain}
\newtheorem{theorem}{\bf Theorem}[section]
\newtheorem*{theorem*}{Theorem 1.1$'$}
\newtheorem{lemma}[theorem]{\bf Lemma}

\theoremstyle{definition}

\theoremstyle{remark}
\newtheorem*{remark*}{\bf Remark}

\newtheorem{example}[theorem]{\bf Example}

\newcommand{\wt}{\widetilde}
\newcommand{\eps}{\varepsilon}

\renewcommand{\[}{\begin{equation}}
\renewcommand{\]}{\end{equation}}

\begin{document}

\title[Spectral density of Hardy kernels]{The spectral density of Hardy kernel matrices}

\author{Alexander Pushnitski}
\address{Department of Mathematics, King’s College London, Strand, London WC2R~2LS, United Kingdom} 
\email{alexander.pushnitski@kcl.ac.uk} 

\date{23 March 2021}

\begin{abstract}
We consider  infinite matrices obtained by restricting Hardy integral kernels to natural numbers. 
For a suitable class of Hardy kernels we describe the absolutely continuous spectrum, the essential spectrum and the asymptotic spectral density of these matrices. 
\end{abstract}

\maketitle

\section{Introduction}\label{sec.a}
Let $k=k(x,y)$ be a complex valued function of two variables $x>0$ and $y>0$, which satisfies two conditions:
\begin{itemize}
\item
$k$ is Hermitian: $k(x,y)=\overline{k(y,x)}$; 
\item
$k$ is homogeneous of degree $-1$: $k(a x,a y)=k(x,y)/a$.
\end{itemize}
We will call $k(x,y)$ a \emph{Hardy kernel}; see remarks below about the terminology. 

The purpose of this paper is to consider some spectral properties of the infinite matrices 
\[
K=\{k(n,m)\}_{n,m=1}^\infty \quad \text{ in $\ell^2(\bbN)$}
\label{eq:3}
\]
as well as their finite truncations
$$
K_N=\{k(n,m)\}_{n,m=1}^N \quad \text{ in $\bbC^N$}
$$
as $N\to\infty$. 
Our point of view is the comparison between the Hardy kernel matrices $K$ and $K_N$ and their ``continuous analogues'', i.e. integral operators $T$ in  $L^2(1,\infty)$, 
$$
Tf(x)=\int_1^\infty k(x,y)f(y)dy, \quad f\in L^2(1,\infty)
$$
as well as their truncations $T_N$ in $L^2(1,N)$, 
$$
T_Nf(x)=\int_1^N k(x,y)f(y)dy, \quad f\in L^2(1,N).
$$
In fact, $T$ and $T_N$ are Wiener-Hopf operators in disguise. Indeed, an exponential change of variable $x=e^u$, $y=e^v$, $u,v\in (0,\infty)$, effects a unitary transformation which transforms $T$ (resp. $T_N$) into the integral operator on $L^2(0,\infty)$ (resp. $L^2(0,\log N)$) with the Wiener-Hopf kernel $k(e^{u-v},1)e^{(u-v)/2}$.

Obviously, such exponential change of variable is not available on integers, 
and therefore the spectral analysis of the Hardy kernel matrices $K$ and $K_N$ presents considerable challenges. 
However, we will show that under suitable restrictions on the kernels:
\begin{itemize}
\item
the essential spectra of $K$ and $T$ coincide;
\item
the absolutely continuous (a.c.) parts of $K$ and $T$ are unitarily equivalent;
\item
the asymptotic spectral density of $K_N$ coincides with that of $T_N$ and is therefore given by the First Szeg\H{o} Limit Theorem;
\item
$K_N$ may have some eigenvalues (in contrast to $T_N$).  
\end{itemize}

To enable meaningful analysis, we need to impose some constraints on the kernels $k$. First observe that the homogeneity of $k$ is equivalent to the representation 
$$
k(x,y)=\frac1{\sqrt{xy}}\wt k(x/y)
$$
with some function $\wt k$ on $(0,\infty)$. In what follows, we shall assume that 
$$
k(x,y)=\frac1{\sqrt{xy}}\frac1{2\pi} \int_{-\infty}^\infty \varphi(t)(x/y)^{-it}dt=\frac1{\sqrt{xy}}\wh \varphi(\log \tfrac{x}{y}), 
$$
where $\varphi\in L^1(\bbR)$ is real-valued and 
$$
\wh\varphi(u)=\frac1{2\pi}\int_{-\infty}^\infty \varphi(t)e^{-itu}dt
$$
is the Fourier transform of $\varphi$.
We shall call $\varphi$ the \emph{symbol} in this context. 
With this notation, after the exponential change of variable $T$ becomes the integral operator with the kernel $\widehat\varphi(u-v)$, in agreement with the standard notion of a symbol of a Wiener-Hopf integral operator. 
We shall henceforth indicate explicitly the dependence on the symbol and denote the above operators by $K(\varphi)$, $K_N(\varphi)$, $T(\varphi)$ and $T_N(\varphi)$.

We finish this section with some remarks on the terminology and the history of the problem; see Section~\ref{sec.exa} for further discussion of related literature.

Hardy kernels are often associated with integral operators on $L^2(0,\infty)$ rather than on $L^2(1,\infty)$. By the same exponential change of variables, such integral operators are unitarily equivalent to the operators of convolution with $\wh\varphi$ on $L^2(\bbR)$. Thus, by the Fourier transform, they are unitarily equivalent to the operators of multiplication by the symbol $\varphi$ in $L^2(\bbR)$, and so the spectral theory of this class of operators is extremely simple.

The most famous Hardy kernel is 
$$
k(x,y)=\frac1{x+y}\ ,
$$
whose study, both in discrete and continuous versions, goes back to Hilbert and Schur. More generally, sufficient conditions for boundedness of Hardy kernel matrices, as well as (sometimes sharp) operator norm bounds for them on $\ell^p$ spaces are discussed in detail in Chapter 9 of the classical monograph \cite{HLP} by Hardy, Littlewood and Polya. They mostly argue by comparing matrices $K$ with the corresponding integral operators on $L^2(0,\infty)$ and impose the condition 
\[
\int_0^\infty \frac{\abs{k(x,1)}}{\sqrt{x}}dx<\infty
\label{eq:14}
\]
together with some monotonicity conditions.

The term ``Hardy kernel'' is usually applied to real-valued kernels which are symmetric, homogeneous of degree $-1$ and satisfy \eqref{eq:14}.
We will not need \eqref{eq:14}, but instead we will impose some conditions on the symbol $\varphi$. 

\section{Main results}\label{sec.b}
We start with some general remarks. 
Let us write the quadratic form of $K_N(\varphi)$ on a vector $a=\{a_n\}_{n=1}^N\in \bbC^N$ as follows:
\begin{align}
\jap{K_N(\varphi)a,a}_{\bbC^N}
&=
\sum_{n,m=1}^N k(n,m)a_n\overline{a}_m
=
\frac1{2\pi}\sum_{n,m=1}^N a_n\overline{a}_m \int_{-\infty}^\infty\varphi(t) n^{-\frac12-it}m^{-\frac12+it}dt
\notag
\\
&=\frac1{2\pi}\int_{-\infty}^\infty\varphi(t) \Abs{\sum_{n=1}^N a_nn^{-\frac12-it}}^2dt\ .
\label{eq:6}
\end{align}
This formula suggests the following:
\begin{itemize}
\item
$K_N(\varphi)$ depends monotonically on $\varphi$ in the quadratic form sense; in particular, if $\varphi\geq0$, then $K_N(\varphi)$ is positive semi-definite. 
\item
The study of $K_N(\varphi)$ is related to the theory of Dirichlet series. 
We will touch upon this aspect of the problem only briefly in Lemma~\ref{lma.8}. 
\item
It is easy to see that $\varphi\in L^1(\bbR)$ is a necessary condition for the definition of $K_N(\varphi)$ to make sense. 
\end{itemize}
As a warm-up, let us compute the asymptotics of the trace of $K_N(\varphi)$. 
We have
$$
k(n,n)=\frac1{n}\widehat\varphi(0)=\frac1{n}\frac1{2\pi}\int_{-\infty}^\infty\varphi(t)dt,
$$
and therefore 
$$
\Tr K_N(\varphi)
=
\frac1{2\pi}\int_{-\infty}^\infty \varphi(t)dt \sum_{n=1}^N \frac1n
=
\frac1{2\pi}\int_{-\infty}^\infty \varphi(t)dt \ \bigl(\log N+O(1)\bigr), \quad N\to\infty.
$$
Now, considering the case $\varphi\geq0$, it follows that for any $\eps>0$ we have 
$$
\#\{j: \lambda_j(K_N(\varphi))>\eps\}
\leq \Tr(K_N(\varphi)/\eps)=O(\log N), \quad N\to\infty,
$$
where $\{\lambda_j(K_N(\varphi))\}_{j=1}^N$ are the eigenvalues of $K_N(\varphi)$ and $\#$ is the number of elements in a given set. So we see that ``most'' of the $N$ eigenvalues of $K_N(\varphi)$ are located near zero and only $O(\log N)$ eigenvalues are located above $\eps>0$. Our first result concerns this logarithmically small proportion of the eigenvalues of $K_N(\varphi)$ and gives their asymptotic density.

\begin{theorem}\label{thm2}
Let $\varphi\in L^1(\bbR)$ be a real-valued symbol, and let $g$ be a Lipschitz continuous function on $\bbR$ with $g(0)=0$. As above, set 
$$
K_N(\varphi)=\{k(n,m)\}_{n,m=1}^N, \quad k(n,m)=\frac1{\sqrt{nm}}\widehat\varphi(\log\tfrac{n}{m}). 
$$
Then 
\[
\lim_{N\to\infty}(\log N)^{-1}\Tr g(K_N(\varphi))=\frac1{2\pi}\int_{-\infty}^\infty g(\varphi(t))dt\ .
\label{eq:1}
\]
\end{theorem}
Since by our assumptions
$$
\abs{g(\varphi(t))}\leq C\abs{\varphi(t)}, 
$$
the integral in the right hand side of \eqref{eq:1} converges absolutely.

Formula \eqref{eq:1} with $T_N(\varphi)$ in place of $K_N(\varphi)$ is well known (after the exponential change of variable reducing $T_N(\varphi)$ to a Wiener-Hopf operator), see e.g.  \cite[Section~8.6]{GS}. It is more commonly used for Toeplitz matrices and in that context it is known as the First Szeg\H{o} Limit Theorem, see e.g. \cite[Section 5.4]{BS}.

As it is standard in this circle of questions, one can replace a Lipschitz function $g$ in \eqref{eq:1} by the characteristic function of an interval $(\lambda,\infty)$, $\lambda>0$, as long as the set $\{t\in\bbR: \varphi(t)=\lambda\}$ has zero Lebesgue measure. This leads to a more expressive formula
$$
\lim_{N\to\infty}(\log N)^{-1}\#\{j: \lambda_j(K_N(\varphi))>\lambda)\}
=
\frac1{2\pi}\meas\{t: \varphi(t)>\lambda\}\ ,
$$
where $\meas$ is the Lebesgue measure on $\bbR$.

Our second result characterises the essential and the absolutely continuous spectra of $K(\varphi)$. 
Here for simplicity we restrict ourselves to bounded positive semi-definite operators. 

\begin{theorem}\label{thm1}
Let $\varphi\in L^\infty(\bbR)$ be a non-negative function satisfying  
\[
\int_{-\infty}^\infty \varphi(t)(\log(2+\abs{t}))^\delta dt<\infty
\label{eq:2}
\]
for some $\delta>2$; assume that $\varphi$ is not identically equal to zero. 
Then $K(\varphi)$, defined by the matrix \eqref{eq:3}, is a bounded positive semi-definite operator on $\ell^2(\bbN)$.
The kernel of $K(\varphi)$ is trivial. The essential spectrum of $K(\varphi)$ coincides with the essential spectrum of $T(\varphi)$, and the a.c. part of $K(\varphi)$ is unitarily equivalent to $T(\varphi)$. 
\end{theorem}

\begin{remark*}
\begin{enumerate}[1.]
\item
By a theorem of M.~Rosenblum \cite{R2}, the spectrum of a  self-adjoint Toeplitz operator is purely a.c. (unless the symbol is constant). The same applies to $T(\varphi)$, since it  is unitarily equivalent to a Toeplitz operator. 
\item
The location of the a.c. spectrum of a Toeplitz operator and its multiplicity function can be explicitly described in terms of the symbol; see \cite{I,R3} and \cite{SY1,SY2}. Therefore, the same applies to $T(\varphi)$. In particular, if the symbol $\varphi\in L^1(\bbR)$ is continuous, then the multiplicity function of $T(\varphi)$ coincides with $1/2$ times the multiplicity function of the operator of multiplication by $\varphi$ in $L^2(\bbR)$. For example, if $\varphi(t)$ is positive,  strictly increasing on $(-\infty,0)$ and strictly decreasing on $(0,\infty)$, then the spectrum of $T(\varphi)$ is $[0,\varphi(0)]$ with multiplicity one. 
\item
It seems to be an interesting open question to determine the class of symbols $\varphi$ that corresponds to  bounded operators $K(\varphi)$; it is unlikely that \eqref{eq:2} is optimal and it is not clear what the optimal condition should be. 
\end{enumerate}
\end{remark*}

Theorem~\ref{thm1} may suggest that the spectrum of $K(\varphi)$ coincides with that of $T(\varphi)$. In fact, this is false; we demonstrate this below by showing that in some natural asymptotic regime, eigenvalues of $K(\varphi)$ always appear (in contrast to $T(\varphi)$). For a self-adjoint operator $A$ let us denote by $n(\lambda;A)$ the rank of the spectral projection $E_A(\lambda,\infty)$. 
In other words, $n(\lambda;A)$ is the number of eigenvalues (counting multiplicities) of $A$ 
in $(\lambda,\infty)$ and $n(\lambda;A)=\infty$, if $A$ has some essential spectrum in $(\lambda,\infty)$. 
\begin{theorem}\label{thm3}
Let $\varphi$ be as in Theorem~\ref{thm1} and for $\alpha>0$, let 
$$
\varphi_\alpha(t)=\frac1\alpha\varphi(t/\alpha).
$$
Then for all sufficiently large $\alpha$ we have 
$$
n(\lambda; K(\varphi_\alpha))\geq\#\{j\in\bbN:\frac1j\widehat\varphi(0)>\lambda\}.
$$
In particular, the number of eigenvalues of $K(\varphi_\alpha)$ above the essential spectrum tends to infinity as $\alpha\to\infty$.
\end{theorem}
\begin{proof}
By min-max, for all $\lambda>0$ we have 
$n(\lambda;K(\varphi_\alpha))\geq n(\lambda; K_N(\varphi_\alpha))$.
By the Riemann-Lebesgue lemma, $\wh \varphi$ tends to zero at infinity and therefore, for all off-diagonal elements of $K_N(\varphi_\alpha)$  we have
$$
k_\alpha(n,m)=\frac1{\sqrt{nm}}\wh\varphi(\alpha\log\tfrac{n}{m})\to0, \quad \alpha\to\infty. 
$$
It follows that $\norm{K_N(\varphi_\alpha)-K_{N,\infty}}\to0$ as $\alpha\to\infty$, where $K_{N,\infty}$ is the diagonal  matrix with elements $\{\wh\varphi(0), \frac12\wh\varphi(0),\dots,\frac1N\wh\varphi(0)\}$ on the diagional. Thus, for any $\eps>0$,
$$
n(\lambda;K_N(\varphi_\alpha))\geq n(\lambda+\eps; K_{N,\infty})
$$
for all sufficiently large $\alpha$. Putting this together and sending  $\eps\to0$ and $N\to\infty$,
we get the desired inequality. 
By Theorem~\ref{thm1}, the essential spectrum of $K(\varphi_\alpha)$ is $[0,\norm{\varphi_\alpha}_{L^\infty}]=[0,\frac1\alpha\norm{\varphi}_{L^\infty}]$; as it shrinks to zero, the number of eigenvalues above it tends to infinity. 
\end{proof}

\begin{remark*}
\begin{enumerate}[1.]
\item
Observe that $T(\varphi)$ satisfies the estimate $\norm{T(\varphi)}\leq\norm{\varphi}_{L^\infty(\bbR)}$. 
In contrast to this, the last theorem shows that the norm of $K(\varphi)$ may be strictly greater than $\norm{\varphi}_{L^\infty(\bbR)}$. Moreover, considering the asymptotics $\alpha\to\infty$, we see that the estimate 
$$
\norm{K(\varphi)}\leq C\norm{\varphi}_{L^\infty(\bbR)}
$$
is false. 
\item
If $\wh\varphi$ tends to zero sufficiently fast at infinity, one can upgrade the above reasoning to conclude that if the eigenvalues of $K(\varphi_\alpha)$ are ordered non-increasingly, then the $j$'th eigenvalue satisfies 
\[
\lambda_j(K(\varphi_\alpha))=\frac1j\wh\varphi(0)+O(1/\alpha), \quad \alpha\to\infty.
\label{eq:13}
\]
This is exactly what was proven in \cite{BPP}, see Example~\ref{exa.2} below. 
\end{enumerate}
\end{remark*}

\section{Examples}\label{sec.exa}
Here we consider several examples of Hardy kernels with references to existing literature. 
\begin{example}
Let 
$$
k(x,y)=\frac1{x+y}, \quad \varphi(t)=\frac{\pi}{\cosh \pi t}. 
$$
This example is very special because in this case $k(x,y)$ is a Hankel kernel, i.e. it depends on the sum $x+y$. 
The corresponding matrix $K(\varphi)$ is the classical Hilbert's matrix, which was explicitly diagonalised by M.~Rosenblum in \cite{R1,R2} in terms of special functions (see also \cite{KS}). 
Rosenblum proved that $K(\varphi)$ has a purely a.c. spectrum $[0,\pi]$ of multiplicity one.

The asymptotic spectral density of the truncated Hilbert matrix was determined by Widom in \cite[Theorem 4.3]{Widom} (note that there is a factor of $2\pi$ missing in \cite{Widom}). See also \cite{Fedele} for an alternative proof and for a more general class of Hankel matrices. 

Our Theorems~\ref{thm2} and \ref{thm1} specialised to this case are in agreement with all of the above but do not add anything new. 

However, the scaled version of this example
$$
k_\alpha(x,y)=\frac1{\sqrt{xy}}\frac1{(x/y)^{\alpha/2}+(y/x)^{\alpha/2}}, 
\quad
\varphi_\alpha(t)=\frac{\pi}{\alpha\cosh(\frac{\pi t}{\alpha})}, \quad \alpha>0,
$$
seems to be new. By Theorem~\ref{thm1}, the a.c. spectrum of $K(\varphi_\alpha)$ is $[0,\pi/\alpha]$ and has multiplicity one. 
\end{example}

\begin{example}\label{exa.2}
For $\alpha>0$, let 
$$
k_\alpha(x,y)
=\frac1{\sqrt{xy}}\min\{(x/y)^\alpha,(y/x)^\alpha\}, 
\quad 
\varphi_\alpha(t)=\dfrac{2}{\alpha(1+(t/\alpha)^2)}.
$$
The operator $K(\varphi_\alpha)$ was introduced in \cite{Brevig} in connection with a question about composition operators on the Hardy space of Dirichlet series. Some estimates for the norm of $K(\varphi_\alpha)$ were given in \cite{Brevig}, and a detailed spectral analysis of this operator was accomplished in \cite{BPP}. 
It was established that $K(\varphi_\alpha)$ has a.c. spectrum $[0,2/\alpha]$ of multiplicity one, no singular continuous spectrum and finitely many eigenvalues above $2/\alpha$, satisfying \eqref{eq:13}. 
These facts are in full agreement with Theorems~\ref{thm1} and \ref{thm3}, as the Wiener-Hopf operator $T(\varphi_\alpha)$ has a purely a.c. spectrum $[0,2/\alpha]$ of multiplicity one.

This example is very special because, as shown in \cite{BPP}, the operator $K(\varphi_\alpha)$ is an inverse of a Jacobi matrix. In particular, since the spectrum of any Jacobi matrix is simple, all eigenvalues of $K(\varphi_\alpha)$ are simple. 
It was also shown that if a Hardy kernel matrix with a continuous kernel is an inverse of a Jacobi matrix, then it coincides, up to a factor, with $K(\varphi_\alpha)$ for some $\alpha>0$. 

The spectral density of $K(\varphi_\alpha)$ was not computed in \cite{BPP}, and Theorem~\ref{thm2} in this case seems to be new. The author is grateful to Uzy Smilansky for asking the question about spectral density in this context. 
\end{example}

\begin{example}
For $\alpha>0$, let 
$$
k_\alpha(x,y)=\frac1{\sqrt{xy}}\biggl(\sqrt{\frac{x}{y}}+\sqrt{\frac{y}{x}}\biggr)^{-\alpha}
=\frac1{\sqrt{xy}}\frac{(xy)^{\alpha/2}}{(x+y)^{\alpha}}, 
\quad
\varphi_\alpha(t)=\frac1{\Gamma(\alpha)}\abs{\Gamma(\tfrac{\alpha}{2}+it)}^2, 
$$
where $\Gamma$ is the Gamma-function. 
We note that the scaling in $\alpha$ here is different from the one in Theorem~\ref{thm3}. 
According to Theorem~\ref{thm1}, the essential spectrum and the a.c. spectrum of $K(\varphi_\alpha)$ is $[0,\varphi_\alpha(0)]$, with multiplicity one. 

For $\alpha=1$ this is the Hilbert matrix, and for $\alpha=2$ this is a variant of the so-called Bergman-Hilbert matrix, considered in \cite{G,DG,KS}. 
More generally, for all $\alpha\in\bbN$, the following matrix was considered in \cite{KS} (as part of a larger three-parameter family of infinite matrices): 
$$
B_\alpha=\{b_{n,m}\}_{n,m=1}^\infty, \quad
b_{n,m}=\frac{\sqrt{(n)_{\alpha-1}(m)_{\alpha-1}}}{(n+m-1)_{\alpha}}, 
$$
where $(x)_\alpha=x(x+1)\cdots(x+\alpha-1)$ is the Pochhammer symbol. 
This is not a Hardy kernel matrix, but for large $n,m$ it has the same asymptotics as $K(\varphi_\alpha)$. 
This matrix was explicitly diagonalised in \cite{KS} in terms of orthogonal polynomials; it was found that the a.c. spectrum of $B_\alpha$ is $[0,\varphi_\alpha(0)]$ with multiplicity one, that it has no singular continuous spectrum and that for large $\alpha$ there are some eigenvalues above the continuous spectrum. 

It is easy to see (cf. the argument of \cite[Proposition 9]{KS}) that $K(\varphi_{2})-B_2$ is a trace class operator, and therefore, by the Kato-Rosenblum theorem, the a.c. parts of $B_2$ and $K(\varphi_{2})$ are unitary equivalent and by Weyl's theorem (invariance of essential spectrum under compact perturbations) the essential spectra of $B_2$ and $K(\varphi_2)$ coincide. This is in agreement with Theorem~\ref{thm1}. 
It is not clear whether the trace class argument works for $\alpha>2$, but in any case comparing Theorem~\ref{thm1} with the results of \cite{KS} we see that the essential spectra and the a.c. spectra of $B_\alpha$ and $K(\varphi_\alpha)$ coincide.  

We are not aware of the spectral density of $K(\varphi_\alpha)$ having been discussed in the literature. 
\end{example}

\begin{example}
Let 
$$
k(x,y)=\frac{\log(x/y)}{x-y}, \quad
\varphi(t)=\frac{\pi^2}{(\cosh \pi t)^2},
$$
and $k(x,x)=1/x$ by continuity.
As far as we are aware, the spectral properties of the corresponding Hardy kernel matrix $K(\varphi)$ have not been considered in the literature, except for the norm bound in \cite[Inequality 342]{HLP} as a ``miscellaneous example''. (The norm bound there is $\pi^2$.)
Theorem~\ref{thm1} in this case ensures that the a.c. spectrum of $K(\varphi)$ is $[0,\pi^2]$ and has multiplicity one.
\end{example}

\begin{example}
A more general version of the previous example is 
$$
k(x,y)=\frac{\omega(\log(x/y))}{x-y}, \quad
\wh \varphi(u)=\frac{\omega(u)}{2\sinh(u/2)}, 
$$
where $\omega$ is a smooth odd function which satisfies the growth condition 
$$
\omega(u)=O(e^{\alpha\abs{u}}), \quad \alpha<1/2,
$$
as $\abs{u}\to\infty$. For example, $\omega(u)=2\sinh(\alpha u)$, $0<\alpha<1/2$ gives
$$
k(x,y)=\frac{(x/y)^\alpha-(y/x)^\alpha}{x-y}, \quad \varphi(t)=\frac{2\pi\sin(2\pi \alpha)}{\cos(2\pi\alpha)+\cosh(2\pi t)} 
$$
with $k(x,x)=2\alpha/x$ by continuity.
\end{example}

\begin{example}
Let 
$$
k(x,y)=\frac1{\sqrt{xy}}\frac{\sin (\log(x/y))}{\log(x/y)}, 
\quad 
\varphi(t)=\pi\chi_{(-1,1)}(t),
$$
and $k(x,x)=1/x$ by continuity. 
This example may be of interest because the symbol here is discontinuous. 
By Theorem~\ref{thm1}, the a.c. spectrum of $K(\varphi)$ is $[0,\pi]$ and has multiplicity one. 
By Theorem~\ref{thm2}, the asymptotic spectral density of $K(\varphi)$ is concentrated at the points $\pi$ and $0$. 
\end{example}

\section{Key notation and the outline of the proof}\label{sec.c}
First we discuss the proof of Theorem~\ref{thm2}. 
In order to motivate what comes next, we factorise $\varphi=\abs{\varphi}^{1/2}\varphi^{1/2}$, where $\varphi^{1/2}=\abs{\varphi}^{1/2}\sign(\varphi)$, and rewrite formula \eqref{eq:6} for the 
the quadratic form of $K_N(\varphi)$ as follows:
$$
\jap{K_N(\varphi)a,a}_{\bbC^N}
=\frac1{2\pi}\int_{-\infty}^\infty
\biggl(\abs{\varphi(t)}^{1/2}\sum_{n=1}^N a_nn^{-\frac12-it}\biggr)
\overline{\biggl(\varphi(t)^{1/2}\sum_{n=1}^N a_nn^{-\frac12-it}\biggr)}
dt \ .
$$
This suggests the following factorisation of $K_N(\varphi)$. 
For $\psi\in L^2(\bbR)$, let us define an operator $A_N(\psi):L^2(\bbR)\to\bbC^N$, 
$$
(A_N(\psi)f)_n=\frac1{\sqrt{2\pi}}\int_{-\infty}^\infty f(t)\psi(t)n^{-\frac12+it}dt\ , \quad n=1,\dots,N.
$$
Then 
$$
K_N(\varphi)=A_N(\varphi^{1/2})A_N(\abs{\varphi}^{1/2})^*.
$$
In a similar way, we factorise the operator $T_N(\varphi)$ 
$$
T_N(\varphi)=B_N(\varphi^{1/2})B_N(\abs{\varphi}^{1/2})^*, 
$$
where $B_N(\psi):L^2(\bbR)\to L^2(1,N)$, 
$$
(B_N(\psi)f)(x)=\frac1{\sqrt{2\pi}}\int_{-\infty}^\infty  f(t)\psi(t)x^{-\frac12+it}dt\ ,\quad  x\in(1,N).
$$
The main step of the proof of Theorem~\ref{thm2} is the proof of \eqref{eq:1} for $g(\lambda)=\lambda^m$, $m\in\bbN$ and for a suitable dense class of symbols $\varphi$. 
By the cyclicity of trace,
\begin{align*}
\Tr \bigl(K_N(\varphi)^m\bigr)
&=\Tr\bigl((A_N(\varphi^{1/2})A_N(\abs{\varphi}^{1/2})^*)^m\bigr)
=\Tr\bigl((A_N(\abs{\varphi}^{1/2})^*A_N(\varphi^{1/2}))^m\bigr)\ ,
\\
\Tr \bigl(T_N(\varphi)^m\bigr)
&=\Tr\bigl((B_N(\varphi^{1/2})B_N(\abs{\varphi}^{1/2})^*)^m\bigr)
=\Tr\bigl((B_N(\abs{\varphi}^{1/2})^*B_N(\varphi^{1/2}))^m\bigr)\ .
\end{align*}
Now let us compare $A_N^*A_N$ and $B_N^*B_N$. Observe that the integral kernel of $A_N(\abs{\varphi}^{1/2})^*A_N(\varphi^{1/2})$ is
$$
\frac1{2\pi}\abs{\varphi}^{1/2}(t_1)\varphi^{1/2}(t_2)\zeta_N(1+i(t_1-t_2)), 
\quad
\zeta_N(s):=\sum_{n=1}^N n^{-s}
$$
and similarly the integral kernel of $B_N(\abs{\varphi}^{1/2})^*B_N(\varphi^{1/2})$ is
$$
\frac1{2\pi}\abs{\varphi}^{1/2}(t_1)\varphi^{1/2}(t_2)\eta_N(1+i(t_1-t_2)), 
\quad
\eta_N(s):=\int_{1}^N x^{-s}dx.
$$
Using elementary estimates of $\zeta_N(s)-\eta_N(s)$, we shall prove
the trace norm estimate
$$
\norm{A_N(\abs{\varphi}^{1/2})^*A_N(\varphi^{1/2})-B_N(\abs{\varphi}^{1/2})^*B_N(\varphi^{1/2})}_{\bS_1}=O(1), \quad N\to\infty,
$$
for a suitable dense subclass of symbols $\varphi$. Here and in what follows $\norm{\cdot}_{\bS_1}$ is the trace norm.
Using this estimate, it is easy to derive the relation 
$$
\Tr \bigl((K_N(\varphi))^m\bigr)=\Tr\bigl((T_N(\varphi))^m\bigr)+O(1), \quad N\to\infty.
$$
The asymptotics of the trace in the right hand side is well known, see e.g. \cite[Section~8.6]{GS}: 
\[
\lim_{N\to\infty}(\log N)^{-1}\Tr\bigl((T_N(\varphi))^m\bigr)
=
\frac1{2\pi}\int_{-\infty}^\infty \varphi(t)^mdt \ .
\label{eq:5}
\]
By linearity, we obtain formula \eqref{eq:1} for polynomials $g$. 
Standard arguments allow us to extend this to all symbols $\varphi\in L^1(\bbR)$ and all Lipschitz continuous functions $g$.

The outline of the proof of Theorem~\ref{thm1} is similar. Recall that here $\varphi\geq0$ by hypothesis. We denote $\psi=\varphi^{1/2}$ and write 
$$
K(\varphi)=A(\psi)A(\psi)^*, \quad
T(\varphi)=B(\psi)B(\psi)^*, 
$$
where $A(\psi):L^2(\bbR)\to\ell^2(\bbN)$ and $B(\psi):L^2(\bbR)\to L^2(0,\infty)$ are defined by 
\begin{align}
(A(\psi)f)_n&=\frac1{\sqrt{2\pi}}\int_{-\infty}^\infty \psi(t) f(t)n^{-\frac12+it}dt\ ,
\label{eq:7}
\\
(B(\psi)f)(x)&=\frac1{\sqrt{2\pi}}\int_{-\infty}^\infty \psi(t) f(t)x^{-\frac12+it}dt \  .
\label{eq:8}
\end{align}
One slight technical complication is that $A(\psi)$ and $B(\psi)$
are not automatically bounded for $\varphi\in L^1(\bbR)$; but they are bounded under the additional assumptions on $\varphi$, listed in the hypothesis of Theorem~\ref{thm1}. 
We prove that under these assumptions, the difference
$$
A(\psi)^*A(\psi)-B(\psi)^*B(\psi)
$$
is a trace class operator. After this, Theorem~\ref{thm1} follows by an application of the Kato-Rosenblum theorem (invariance of a.c. spectrum under trace class perturbations) and Weyl's theorem (invariance of essential spectrum under compact perturbations).

\section{Proof of Theorem~\ref{thm2}}\label{sec.d}

For $\psi\in L^2(\bbR)$, let $A_N(\psi):L^2(\bbR)\to\bbC^N$ and $B_N(\psi):L^2(\bbR)\to L^2(1,N)$ be the operators defined in the previous section. 
\begin{lemma}\label{lma.1}
Let $\varphi\in L^1(\bbR)$ satisfy 
\[
\int_{-\infty}^\infty \abs{\varphi(t)}(\log(2+\abs{t}))^\delta dt<\infty
\label{eq:4}
\]
for some $\delta>2$. Let 
\[
D_N(\varphi)=A_N(\abs{\varphi}^{1/2})^*A_N(\varphi^{1/2})-B_N(\abs{\varphi}^{1/2})^*B_N(\varphi^{1/2}). 
\label{eq:9}
\]
Then $D_N(\varphi)$ converges  in the trace norm as $N\to\infty$; in particular, $\norm{D_N(\varphi)}_{\bS_1}=O(1)$. 
\end{lemma}
\begin{proof}
For $x\geq1$, denote 
$$
u_x(t)=\varphi^{1/2}(t)x^{-it}, \quad
v_x(t)=\abs{\varphi}^{1/2}(t)x^{-it}. 
$$
We consider $u_x$ and $v_x$ as elements of $L^2(\bbR)$; observe that 
$$
\norm{v_x}^2_{L^2(\bbR)}=\norm{u_x}^2_{L^2(\bbR)}=\norm{\varphi}_{L^1(\bbR)}.
$$
Below $\jap{\cdot,\cdot}$ is the inner product in $L^2(\bbR)$. 
With this notation we can write
\begin{align*}
A_N(\abs{\varphi}^{1/2})^*A_N(\varphi^{1/2})&=
\frac1{2\pi}\sum_{n=1}^N \frac1n \jap{\cdot,u_n}v_n, 
\\
B_N(\abs{\varphi}^{1/2})^*B_N(\varphi^{1/2})&=
\frac1{2\pi}\int_1^N \frac1x\jap{\cdot,u_x}v_x \ dx
\\
&=\frac1{2\pi}\sum_{n=2}^N\int_{n-1}^n \frac1x\jap{\cdot,u_x}v_x\  dx\ .
\end{align*}
Our aim is to estimate the trace norm of the difference of $n$'th terms in the two sums in the right hand sides. For $n\geq2$ we have 
\begin{align*}
\frac1n\jap{\cdot,u_n}v_n-&\int_{n-1}^n \frac1x\jap{\cdot,u_x}v_x\ dx
=\int_{n-1}^n \biggl(\frac1n-\frac1x\biggr)\jap{\cdot,u_n}v_n\ dx
\\
&+\int_{n-1}^n \frac1x\jap{\cdot,u_n-u_x}v_n\ dx
+\int_{n-1}^n\frac1x\jap{\cdot,u_x}(v_n-v_x)\ dx\ .
\end{align*}
The first term is easy to estimate:
$$
\Norm{\int_{n-1}^n \biggl(\frac1n-\frac1x\biggr)\jap{\cdot,u_n}v_n\ dx}_{\bS_1}
\leq 
\frac1{n(n-1)}\norm{\jap{\cdot,u_n}v_n}_{\bS_1}=\frac1{n(n-1)}\norm{\varphi}_{L^1(\bbR)}.
$$
In order to estimate the second and third terms, we need to consider the differences $u_n-u_x$ and $v_n-v_x$. We have
$$
u_n(t)-u_x(t)=\varphi^{1/2}(t)(n^{-it}-x^{-it})=\varphi^{1/2}(t)(e^{-it\log n}-e^{-it\log x})\ .
$$
We use the elementary estimate $\abs{e^{ia}-1}\leq \min\{\abs{a},2\}$. 
Then, for $n-1\leq x\leq n$, we have 
$$
\abs{u_n(t)-u_x(t)}
\leq 
\abs{\varphi}^{1/2}(t)\min\{\abs{t}\Abs{\log \tfrac{x}{n}}, 2\}
\leq 
2\abs{\varphi}^{1/2}(t)\min\{\tfrac{\abs{t}}{n}, 1\},
$$
because $\Abs{\log \tfrac{x}{n}}\leq \Abs{\log (1-\frac1n)}\leq 2/n$. 
It follows that 
$$
\norm{u_n-u_x}_{L^2(\bbR)}^2
\leq
4\int_{-\infty}^\infty \abs{\varphi(t)}\min\{\tfrac{t^2}{n^2}, 1\}dt
=
4\int_{-\infty}^\infty \abs{\varphi(t)}(\log(2+\abs{t}))^\delta F_n(t)dt,
$$
where 
$$
F_n(t)=(\log(2+\abs{t}))^{-\delta}\min\{\tfrac{t^2}{n^2}, 1\}. 
$$
Elementary considerations show that 
$$
F_n(t)\leq C(\log n)^{-\delta},
$$
and therefore 
$$
\norm{u_n-u_x}_{L^2(\bbR)}^2
\leq
C(\log n)^{-\delta}\int_{-\infty}^\infty \abs{\varphi(t)}(\log(2+\abs{t}))^\delta dt.
$$
Of course, exactly the same estimate holds for $v_n-v_x$. 
Putting this together, we find
$$
\Norm{\frac1n\jap{\cdot,u_n}v_n-\int_{n-1}^n \frac1x\jap{\cdot,u_x}v_x\ dx}_{\bS_1}
\leq C(\varphi)(n^{-2}+n^{-1}(\log n)^{-\delta/2}). 
$$
Since by assumption $\delta>2$, it follows that the operator $D_N(\varphi)$ is represented as a partial sum of a series that converges absolutely in trace norm. This yields the required statement. 
\end{proof}

\begin{lemma}\label{lma.2}
Let $\psi\in L^\infty(\bbR)$; then the operator $B_N(\psi)$ satisfies the operator norm estimate
$$
\norm{B_N(\psi)}\leq \norm{\psi}_{L^\infty}. 
$$
\end{lemma}
\begin{proof}
We have 
$$
(B_N(\psi)f)(x)=\sqrt{2\pi}\frac1{\sqrt{x}}\widehat{\psi f}(-\log x), 
$$
and therefore
\begin{align*}
\int_1^N\abs{(B_N(\psi)f)(x)}^2dx
&=
2\pi\int_1^N\abs{\widehat{\psi f}(-\log x)}^2\frac{dx}{x}
=
2\pi\int_0^{\log N}\abs{\widehat{\psi f}(-t)}^2dt
\\
&\leq
2\pi \int_{-\infty}^\infty \abs{\widehat{\psi f}(t)}^2dt
=
\norm{\psi f}_{L^2(\bbR)}^2
\leq 
\norm{\psi}_{L^\infty(\bbR)}^2\norm{f}_{L^2(\bbR)}^2,
\end{align*}
as required.
\end{proof}

\begin{lemma}\label{lma.3}
Let $\varphi\in L^\infty(\bbR)$ satisfy condition \eqref{eq:4} for some $\delta>2$. Then 
the operator $A_N(\abs{\varphi}^{1/2})$ satisfies the operator norm estimate 
$\norm{A_N(\abs{\varphi}^{1/2})}=O(1)$ as $N\to\infty$. 
\end{lemma}
\begin{proof}
Of course, it suffices to prove the statement for $\varphi\geq0$. In this case, 
combining the results of two previous Lemmas, we find
\begin{align*}
\norm{A_N(\abs{\varphi}^{1/2})}^2
&=
\norm{A_N(\abs{\varphi}^{1/2})^*A_N(\abs{\varphi}^{1/2})}
\\
&\leq
\norm{D_N(\varphi)}+\norm{B_N(\abs{\varphi}^{1/2})^*B_N(\abs{\varphi}^{1/2})}
=O(1)
\end{align*}
as $N\to\infty$. 
\end{proof}

\begin{lemma}\label{lma.4}
Let $\varphi\in L^\infty(\bbR)$ satisfy  \eqref{eq:4} for some $\delta>2$.
Then for any $m\in\bbN$ we have 
$$
\lim_{N\to\infty}(\log N)^{-1}\Tr\bigl((K_N(\varphi))^m\bigr)
=
\frac1{2\pi}\int_{-\infty}^\infty \varphi(t)^mdt \ .
$$
\end{lemma}
\begin{proof}
The case $m=1$ is a direct calculation (see Section~\ref{sec.b}); below we assume $m\geq2$. 
Using the cyclicity of trace, we find 
\begin{align*}
\Tr \bigl((K_N(\varphi))^m\bigr)&-\Tr \bigl((T_N(\varphi))^m\bigr)
\\
&=
\Tr\bigl((A_N(\varphi^{1/2})A_N(\abs{\varphi}^{1/2})^*)^m\bigr)
-
\Tr\bigl((B_N(\varphi^{1/2})B_N(\abs{\varphi}^{1/2})^*)^m\bigr)
\\
&=
\Tr\bigl((A_N(\varphi^{1/2})^*A_N(\abs{\varphi}^{1/2}))^m\bigr)
-
\Tr\bigl((B_N(\varphi^{1/2})^*B_N(\abs{\varphi}^{1/2}))^m\bigr)\ .
\end{align*}
Denoting 
$$
X=A_N(\varphi^{1/2})^*A_N(\abs{\varphi}^{1/2}), 
\quad
Y=B_N(\varphi^{1/2})^*B_N(\abs{\varphi}^{1/2}), 
$$
we have 
$$
X^m-Y^m=(X-Y)X^{m-1}+\cdots+Y^{m-1}(X-Y),
$$
and therefore
$$
\abs{\Tr(X^m-Y^m)}\leq \norm{X-Y}_{\bS_1}(\norm{X}^{m-1}+\cdots+\norm{Y}^{m-1}). 
$$
By Lemma~\ref{lma.1}, we have $\norm{X-Y}_{\bS_1}=O(1)$ and by Lemmas~\ref{lma.2} and \ref{lma.3} the sum in the brackets above is also $O(1)$ as $N\to\infty$. We conclude that 
$$
\Tr (K_N(\varphi)^m)-\Tr (T_N(\varphi)^m)=O(1), \quad N\to\infty. 
$$
Using \eqref{eq:5}, we conclude the proof. 
\end{proof}

The rest of the proof of Theorem~\ref{thm2} is a standard approximation argument. 

\begin{lemma}\label{lma.5}
Let $g$ be a Lipschitz continuous function with $g(0)=0$, and let $\varphi\in L^\infty(\bbR)$ be a real-valued symbol with compact support. Then the asymptotic formula \eqref{eq:1} holds true. 
\end{lemma}
\begin{proof}
First let us choose $\Lambda>0$ such that $\norm{\varphi}_{L^\infty(\bbR)}\leq \Lambda$ and $\norm{K_N(\varphi)}\leq \Lambda$ for all $N$ (this can be done by Lemma~\ref{lma.3}). Next, let $R>0$ be such that $\supp\varphi\subset[-R,R]$. 

It suffices to consider the case of real-valued $g$. 
Let $\eps>0$ be given. Using the Weierstrass approximation theorem, it is not difficult to construct two polynomials $p_+$ and $p_-$ with real coefficients such that $p_+(0)=p_-(0)=0$, 
$$
p_-(\lambda)\leq g(\lambda)\leq p_+(\lambda), \quad \abs{\lambda}\leq \Lambda
$$
and
\[
0\leq p_+(\lambda)-p_-(\lambda)\leq \eps, \quad \abs{\lambda}\leq \Lambda.
\label{eq:11}
\]
Then 
$$
\Tr p_-(K_N(\varphi))\leq \Tr g(K_N(\varphi))\leq \Tr p_+(K_N(\varphi)). 
$$
By taking linear combinations of monomials in Lemma~\ref{lma.4}, we find that \eqref{eq:1} holds for all polynomials vanishing at zero, and therefore
\begin{align*}
\limsup_{N\to\infty}(\log N)^{-1}\Tr g(K_N(\varphi))
\leq 
\limsup_{N\to\infty}(\log N)^{-1}\Tr p_+(K_N(\varphi))
\\
=
\frac1{2\pi}\int_{-\infty}^\infty p_+(\varphi(t))dt
=
\frac1{2\pi}\int_{-R}^R p_+(\varphi(t))dt
\end{align*}
and similarly 
$$
\liminf_{N\to\infty}(\log N)^{-1}\Tr g(K_N(\varphi))
\geq 
\frac1{2\pi}\int_{-R}^R p_-(\varphi(t))dt.
$$
On the other hand, by \eqref{eq:11} we have
$$
\int_{-R}^R (p_+(\varphi(t))-p_-(\varphi(t)))dt\leq 2R\eps. 
$$
Since $\eps>0$ is arbitrary, we find that 
$$
\limsup_{N\to\infty}(\log N)^{-1}\Tr g(K_N(\varphi))
=
\liminf_{N\to\infty}(\log N)^{-1}\Tr g(K_N(\varphi))
=
\frac1{2\pi}\int_{-R}^R g(\varphi(t))dt,
$$
as required. 
\end{proof}

\begin{proof}[Proof of Theorem~\ref{thm2}]
Throughout the proof, we fix a Lipschitz function $g$ with $g(0)=0$ and we denote by $\norm{g}_{\text{Lip}}$ the norm of $g$ in the Lipschitz class. 
Our task is to extend \eqref{eq:1} from compactly supported bounded symbols to all real-valued symbols in $L^1(\bbR)$. For $\varphi=\overline{\varphi}\in L^1(\bbR)$
 we denote 
$$
M_N(\varphi)=(\log N)^{-1}\Tr g(K_N(\varphi)), 
\quad
M_\infty(\varphi)=\frac1{2\pi}\int_{-\infty}^\infty g(\varphi(t))dt.
$$
Further, we set 
$$
\overline{M}(\varphi)=\limsup_{N\to\infty}M_N(\varphi), \quad
\underline{M}(\varphi)=\liminf_{N\to\infty}M_N(\varphi).
$$
Now \eqref{eq:1} is equivalent to  
$$
\overline{M}(\varphi)=\underline{M}(\varphi)=M_\infty(\varphi). 
$$
Let us discuss the continuity of the (nonlinear) functionals $\overline{M}(\varphi)$, $\underline{M}(\varphi)$, $M_\infty(\varphi)$ with respect to $\varphi\in L^1(\bbR)$. 
Below $\varphi_1$, $\varphi_2$ are two real-valued symbols. By the Lipschitz continuity of $g$, we have
$$
\abs{M_\infty(\varphi_2)-M_\infty(\varphi_1)}\leq \frac1{2\pi}\norm{g}_{\text{Lip}}\norm{\varphi_2-\varphi_1}_{L^1(\bbR)}, 
$$
and so $M_\infty(\varphi)$ is Lipschitz continuous in $\varphi\in L^1(\bbR)$. 

Next, suppose $\varphi_1\leq\varphi_2$ for a.e. $t\in\bbR$; then by \eqref{eq:6} we have
$$
K_N(\varphi_1)\leq K_N(\varphi_2)
$$
in the quadratic form sense. By min-max, it follows that for all eigenvalues of $K_N(\varphi_1)$ and $K_N(\varphi_2)$ (labelled in non-decreasing order and counted with multiplicities) we have
$$
\lambda_j(K_N(\varphi_1))\leq \lambda_j(K_N(\varphi_2)), \quad j=1,\dots,N
$$
and therefore for all $j$
\begin{multline*}
\abs{g(\lambda_j(K_N(\varphi_2)))-g(\lambda_j(K_N(\varphi_1)))}
\\
\leq
\norm{g}_{\text{Lip}}\abs{\lambda_j(K_N(\varphi_2))-\lambda_j(K_N(\varphi_1))}
=
\norm{g}_{\text{Lip}}(\lambda_j(K_N(\varphi_2))-\lambda_j(K_N(\varphi_1))).
\end{multline*}
Summing over $j$, we obtain
\begin{multline*}
\abs{\Tr g(K_N(\varphi_2))-\Tr g(K_N(\varphi_1))}
\leq 
\norm{g}_{\text{Lip}}(\Tr K_N(\varphi_2)-\Tr K_N(\varphi_1))
\\
=
(\log N)\norm{g}_{\text{Lip}}\frac1{2\pi}\int_{-\infty}^\infty (\varphi_2(t)-\varphi_1(t))dt
=
(\log N)\norm{g}_{\text{Lip}}\frac1{2\pi} \norm{\varphi_2-\varphi_1}_{L^1(\bbR)}\ .
\end{multline*}
It follows that for $\varphi_1\leq\varphi_2$ we have
$$
\abs{M_N(\varphi_2)-M_N(\varphi_1)}\leq \frac1{2\pi}\norm{g}_{\text{Lip}}\norm{\varphi_2-\varphi_1}_{L^1(\bbR)}\ .
$$
Taking upper and lower limits, we finally conclude that 
\begin{align*}
\abs{\overline{M}(\varphi_2)-\overline{M}(\varphi_1)}
&\leq \frac1{2\pi}\norm{g}_{\text{Lip}}\norm{\varphi_2-\varphi_1}_{L^1(\bbR)}, 
\\
\abs{\underline{M}(\varphi_2)-\underline{M}(\varphi_1)}
&\leq \frac1{2\pi}\norm{g}_{\text{Lip}}\norm{\varphi_2-\varphi_1}_{L^1(\bbR)}, 
\end{align*}
for $\varphi_1\leq\varphi_2$. Of course, the same is true for $\varphi_1\geq\varphi_2$. 
Thus, the functionals $\overline{M}(\varphi)$ and $\underline{M}(\varphi)$ are Lipschitz continuous with respect to monotone convergence of $\varphi$ in $L^1(\bbR)$. 

Now it remains to approximate a given $\varphi\in L^1(\bbR)$ by compactly supported bounded symbols while using monotone convergence. This is an easy exercise, we leave out the details. 
\end{proof}

\section{Proof of Theorem~\ref{thm1}}\label{sec.е}
Let $\varphi$ be as in the hypothesis of the theorem and let $\psi=\varphi^{1/2}$. 
Our first task is to prove that the operators $A(\psi)$ and $B(\psi)$, formally defined by \eqref{eq:7} and \eqref{eq:8}, are well-defined and bounded.
For $B(\psi)$ this is an easy task, given by the same calculation as in the proof of Lemma~\ref{lma.2}:
\begin{align*}
\int_1^\infty\abs{(B(\psi)f)(x)}^2dx
&=
2\pi\int_1^\infty\abs{\widehat{\psi f}(-\log x)}^2\frac{dx}{x}
=
2\pi\int_0^{\infty}\abs{\widehat{\psi f}(-t)}^2dt
\\
&\leq
\norm{\psi f}_{L^2(\bbR)}^2
\leq 
\norm{\psi}_{L^\infty(\bbR)}^2\norm{f}_{L^2(\bbR)}^2,
\end{align*}
and so $B(\psi)$ is bounded if $\psi$ is bounded. 

\begin{lemma}\label{lma.6}
The operator $A(\psi):L^2(\bbR)\to\ell^2(\bbN)$ is well defined by \eqref{eq:7} and bounded. On the set of finitely supported elements $a\in\ell^2(\bbN)$, 
the adjoint is given by the formula
$$
(A(\psi)^*a)(t)=\frac1{\sqrt{2\pi}}\psi(t)\sum_{n=1}^\infty a_n n^{-\frac12-it}
$$
for a.e. $t\in\bbR$. 
The factorisation
$$
K(\varphi)=A(\psi)A(\psi)^*
$$
holds true in the sense that on any finitely supported elements $a,b\in\ell^2(\bbN)$, we have
$$
\jap{K(\varphi)a,b}_{\ell^2(\bbN)}=\jap{A(\psi)^*a,A(\psi)^*b}_{L^2(\bbR)}. 
$$
In particular, the operator $K(\varphi)$, defined initially on finitely supported elements,
 extends to a bounded positive semi-definite operator on $\ell^2(\bbN)$. 
\end{lemma}
\begin{proof}
The sequence $(A(\psi)f)_n$, $n\in\bbN$, is clearly well defined. By Lemma~\ref{lma.3}, there exists a constant $C>0$ independent of $N$, such that for all $f\in L^2(\bbR)$, 
$$
\sum_{n=1}^N \abs{(A(\psi)f)_n}^2\leq C\norm{f}^2.
$$
It follows that $A(\psi)f\in\ell^2(\bbN)$ and the operator $A(\psi)$ is bounded. 
Computing the adjoint and checking the factorisation of $K(\varphi)$ is a direct calculation. 
\end{proof}

The following lemma is the only point in the paper where we use some results from the theory of Hardy spaces of Dirichlet series. 
\begin{lemma}\label{lma.8}
The kernel of $K(\varphi)$ is trivial. 
\end{lemma}
\begin{proof}
First let $a\in\ell^2(\bbN)$ and let $f=f(s)$ be the corresponding Dirichlet series
$$
f(s)=\sum_{n=1}^\infty a_n n^{-s}, \quad \Re s>1/2.
$$
The space of all such functions $f$ is known as the Hardy space of Dirichlet series $\mathscr{H}^2$. 
It is known (see e.g. \cite[Theorem 4.11]{HLS}) that $\mathscr{H}^2$ is embedded into the locally uniform Hardy space in the half-plane $\Re s>1/2$, i.e. 
$$
\sup_{\tau\in\bbR}\sup_{\sigma>1/2}\int_{\tau}^{\tau+1}\abs{f(\sigma+it)}^2dt\leq C\norm{a}_{\ell^2(\bbN)}^2.
$$
In particular,  the boundary values $f(\frac12+it):=\lim\limits_{\eps\to0_+}f(\frac12+\eps+it)$ exist and are non-zero for a.e. $t\in\bbR$. Moreover, if 
$$
f_N(s)=\sum_{n=1}^N a_n n^{-s}, 
$$
then for every $\tau\in\bbR$, 
\[
\lim_{N\to\infty}
\int_{\tau}^{\tau+1}\abs{f(\tfrac12+it)-f_N(\tfrac12+it)}^2dt=0.
\label{eq:12}
\]
Suppose $A^*(\psi)a=0$; let us prove that $a=0$. 
For $a^{(N)}=(a_1,\dots,a_N,0,\dots)$ we have 
$\norm{A^*(\psi)a^{(N)}}_{L^2(\bbR)}\to\norm{A^*(\psi)a}_{L^2(\bbR)}=0$ as $N\to\infty$. 
By the formula for the adjoint of $A(\psi)$ in the previous lemma, this means that 
$$
\lim_{N\to\infty}\int_{-\infty}^\infty\abs{\psi(t)}^2\Abs{\sum_{n=1}^N a_nn^{-\frac12-it}}^2dt
=\lim_{N\to\infty}\int_{-\infty}^\infty \abs{\psi(t)}^2\abs{f_N(\tfrac12+it)}^2dt
=0.
$$
Combining this with \eqref{eq:12}, we find that $\psi(t)f(\frac12+it)=0$ for a.e. $t\in\bbR$. 
Since by our assumption $\psi$ is not identically zero, we find that $f(\frac12+it)=0$ on a set of a positive measure; hence $f$ must vanish identically and so $a=0$. 
\end{proof}

\begin{lemma}\label{lma.7}
The operator 
$$
D(\varphi)=A(\psi)^*A(\psi)-B(\psi)^*B(\psi)
$$
is trace class. 
\end{lemma}
\begin{proof}
Let us assume that $\bbC^N$ is embedded in $\ell^2(\bbN)$ in a natural way; then the operator $A_N(\psi)$ can be regarded as the operator from $L^2(\bbR)$ to $\ell^2(\bbN)$. In the same way, $B_N(\psi)$ can be regarded as an operator from $L^2(\bbR)$ to $L^2(1,\infty)$. 
From the boundedness of $A(\psi)$ and $B(\psi)$ it is clear that we have the convergence $A_N(\psi)\to A(\psi)$ and $B_N(\psi)\to B(\psi)$ in the strong operator topology as $N\to\infty$. 
It follows that $D_N(\varphi)\to D(\varphi)$ in the weak operator topology; here $D_N(\varphi)$ is defined in \eqref{eq:9}. From Lemma~\ref{lma.1} we know that $D_N(\varphi)$ converges to a trace class operator; therefore, by the uniqueness of the weak limit, $D(\varphi)$ is trace class.  
\end{proof}

\begin{proof}[Proof of Theorem~\ref{thm1}]
We first note that by Weyl's theorem about the invariance of essential spectrum under compact perturbations, the essential spectra of $A(\psi)^*A(\psi)$ and $B(\psi)^*B(\psi)$ coincide. Similarly, by the Kato-Rosenblum theorem on trace class perturbations, the a.c. parts of $A(\psi)^*A(\psi)$ and $B(\psi)^*B(\psi)$ are unitarily equivalent. 
Next, we recall that we have already proved that 
$$
K(\varphi)=A(\psi)A(\psi)^*, \quad
T(\varphi)=B(\psi)B(\psi)^*.
$$
It is well known that for any bounded operator $X$ in a Hilbert space, the operators 
$$
X^*X|_{(\Ker X)^\perp} \text{ and } XX^*|_{(\Ker X^*)^\perp}
$$
are unitarily equivalent.
We conclude that the essential spectra of $K(\varphi)|_{(\Ker K(\varphi))^\perp}$ and $T(\varphi)|_{(\Ker T(\varphi))^\perp}$ coincide. Since the kernels of both $T(\varphi)$ and $K(\varphi)$ are trivial, we find that the essential spectra of $K(\varphi)$ and $T(\varphi)$ coincide. 
Similarly, the a.c. parts of $K(\varphi)$ and $T(\varphi)$ are unitarily equivalent; since $T(\varphi)$ is purely a.c., we find that the a.c. part of $K(\varphi)$ is unitarily equivalent to $T(\varphi)$. 
\end{proof}

\section*{Acknowledgements}
The author is grateful to Uzy Smilansky for asking the question about the spectral density in Example~\ref{exa.2} and to Ole Brevig,  Nikolai Nikolski, Eugene Shargorodsky and Alexander Sobolev for useful discussions.

\end{document}